\documentclass[a4paper,12pt]{amsart}
\usepackage[utf8]{inputenc}
\usepackage[T1]{fontenc}
\usepackage{url}
\usepackage{amssymb} 

\usepackage{hyperref}
\usepackage{mathtools}

\theoremstyle{plain}
\newtheorem{thm}{Theorem}[section]

\newtheorem{lem}[thm]{Lemma}

\theoremstyle{definition}

\theoremstyle{remark}
\newtheorem{rem}[thm]{Remark}

\numberwithin{equation}{section}

\DeclareMathOperator{\Lip}{Lip}

\DeclareMathOperator{\N}{\mathbb{N}}
\DeclareMathOperator{\image}{Im}

\title[Attaining diameter two properties in Lipschitz-free spaces]{A note on attaining diameter two properties in Lipschitz-free spaces}

\author{Jaan Kristjan Kaasik}

\address{Institute of Mathematics and Statistics, University of Tartu, Narva mnt 18, 51009, Tartu, Estonia}%

\email{jaan.kristjan.kaasik@ut.ee}


\date{}

\begin{document}

\begin{abstract}
We prove that in Lipschitz-free spaces the strong diameter two property, the diameter two property, and the local diameter two property coincide with their corresponding attaining variants.
\end{abstract}

\maketitle
\section{Introduction}

Let $M$ be a metric space with a distinguished point $0\in M$. We denote by $\Lip_0(M)$ the Banach space of all real-valued Lipschitz functions vanishing at $0$, equipped with the norm
\[
\|f\|=\sup\left\{\frac{|f(p)-f(q)|}{d(p,q)}\colon p,q\in M, p\neq q\right\}.
\]
Let $\delta:M\rightarrow \Lip_0(M)^*$ denote the canonical embedding, defined by 
\[
    \langle f, \delta(p)\rangle=f(p),
\qquad p\in M,\ f\in \Lip_0(M).\]
It is well known that $\operatorname{Lip}_0(M)$ is a dual space whose canonical predual is \emph{the Lipschitz-free space $\mathcal{F}(M)$} defined by
\[
\mathcal F(M)\coloneqq \overline{\operatorname{span}}\,\delta(M)\subset \operatorname{Lip}_0(M)^\ast.
\] 
For $p,q\in M$ with $p\neq q$, 
we denote by 
\[m_{p,q}=\frac{\delta(p)-\delta(q)}{d(p,q)}\] 
a norm one elementary molecule in $\mathcal{F}(M)$.
For background, we refer to \cite{MR2030906} and \cite{MR3792558}.

In these spaces we are interested in the following properties. We say that a Banach space $X$ has the 
\emph{Daugavet property }, if
    for every $x\in S_X$, every slice $S$ of $B_X$, and every $\varepsilon>0$, there exists $y\in S$ such that 
    \[\|x-y\|\geq 2-\varepsilon.\]
Following \cite{MR3098474}, a Banach space $X$ has \emph{the local diameter two property (LD$2$P), the diameter two property (D$2$P), or the strong diameter two property (SD$2$P)} if  every slice of $B_X$, every nonempty weakly open subset of $B_X$, or every convex combination of slices of $B_X$, has diameter two, respectively.
In general,
\[\text{Daugavet property}\Rightarrow \text{SD2P}\Rightarrow \text{D2P}\Rightarrow \text{LD2P},\] and the implications are strict.

One may also consider the corresponding \textit{attaining} variants (ALD$2$P, AD$2$P, and ASD$2$P), where the diameter two is required to be attained. Analogous implications hold for these properties and they are strictly stronger than their non-attaining counterparts (see  \cite{MR4401984} and \cite{MR4787365}).

In Lipschitz-free space, diameter two properties are remarkably rigid. It follows from \cite{{MR3912824},MR3794100,MR4522807,MR2379289} that for a complete metric space $M$, the following assertions are equivalent:

\begin{itemize}
        \item $M$ is a length space;
        \item $\mathcal F(M)$ has the Daugavet property;
        \item $\mathcal F(M)$ has the SD$2$P;
        \item $\mathcal F(M)$ has the D$2$P;
        \item $\mathcal F(M)$ has the LD$2$P;
        \item $\mathcal F(M)$ is locally almost square;
        \item the unit ball of $\mathcal F(M)$ does not have strongly exposed points.
\end{itemize}
It is therefore natural to ask whether the attaining variants also fall under the same characterisation. We answer this affirmatively.

\begin{thm}\label{thm:main}
    Let $M$ be a complete metric space. The following assertions are equivalent:
    \begin{enumerate}
        \item $M$ is a length space;
        \item $\mathcal{F}(M)$ has the Daugavet property;
        \item $\mathcal F(M)$ has the SD$2$P;
        \item $\mathcal F(M)$ has the D$2$P;
        \item $\mathcal F(M)$ has the LD$2$P;
        \item $\mathcal F(M)$ has the ASD$2$P;
        \item $\mathcal F(M)$ has the AD$2$P;
        \item $\mathcal F(M)$ has the ALD$2$P;
    \end{enumerate}
\end{thm}
The equivalence of $(1)$-$(5)$ is known from previously cited results. Thus, it remains to show that if $M$ is a length space, then $\mathcal{F}(M)$ has the ASD$2$P. We prove this in Theorem \ref{thm:ASD2P_main}. 

It is an open question whether every Banach space with the Daugavet property also has the ASD$2$P (see \cite[Question 6.10]{MR4401984}). Theorem \ref{thm:main} provides a partial positive answer to this question for Lipschitz-free spaces.

One may also consider other attaining variants of related properties, in which the diameter is required to be attained. One such property is the \emph{perfect Daugavet property}, requiring that the distance two in the definition of the Daugavet property is attained (see \cite[page 212]{10481/104200}). It would be interesting to know whether there exists a Lipschitz-free space with this property. 

We note that a separable example cannot exist, since separable Lipschitz-free spaces already lack the attaining diametral local diameter two property (see \cite{MR3818544} for background on diametral diameter two properties). Indeed, if $M$ is separable and $\{p_n:n\in \N\}\subset M\setminus \{0\}$ is dense in $M$, then the element 
\[\mu=\sum_{n=1}^\infty 2^{-n}m_{p_n,0}\in S_{\mathcal{F}(M)}\]
admits a unique norming functional $d(\cdot,0)$. Consider the slice $S$ determined by the functional $d(\cdot,0)$ and $1$. If $\mathcal{F}(M)$ would have the perfect Daugavet property, then there exist $\nu\in S$ and $f\in S_{\Lip_0(M)}$ so that $\langle f,\mu \rangle=1$ and $\langle f,\nu \rangle=-1$. But it must hold that $f(\cdot)=d(\cdot,0)$, which contradicts $\nu\in S$.

Analogous questions arise for spaces of Lipschitz functions. As these are dual Banach spaces, one can also consider weak$^*$-slices and weak$^*$-open sets, and define the corresponding weak$^*$-diameter two properties.  Previous work in spaces of Lipschitz functions has provided characterisations of weak$^*$-diameter two properties  (see \cite{HALLER2026130178}, \cite{MR4026495}, \cite{MR3803112}) and established necessary conditions for diameter two properties (see \cite{MR4849357}). To the best of our knowledge, the attaining variants in these spaces have not been investigated.

\subsection*{Notation}  We only consider real Banach spaces. For a Banach space $X$ we denote the closed unit ball by $B_X$, the unit sphere by $S_X$, and the dual space by $X^*$. Given $x^*\in S_{X^*}$ and $\alpha>0$, the slice of $B_X$, determined by $x^*$ and $\alpha$, is 
\[
S(x^*,\alpha)=\{x\in B_X: x^*(x)>1-\alpha\}.
\]

Let $M$ be a metric space. A continuous mapping $\gamma\colon[0,1]\rightarrow M$ is called a path joining $\gamma(0)$ and $\gamma(1)$. Its length is defined by
\[
L(\gamma) = \sup\left\{\sum_{i=0}^{n-1} d\big(\gamma(t_i), \gamma(t_{i+1})\big) \colon n\in \mathbb{N},\, 0=t_0\leq \dotsb\leq t_n=1\right\}.
\]
The metric space $M$ is a length space if for every $p,q\in M$, 
\[
d(p,q)=\inf\{L(\gamma):\gamma \text{ is a path joining } p \text{ and }q\}.\]

Given $p$ in $M$ and $r>0$, we denote by $B(p,r)$ the open ball in $M$ centred at $p$ of radius $r$.

\section{Proof of main theorem}
We begin with an auxiliary lemma.
\begin{lem}\label{lemma:molecule_in_slice}Let $M$ be a length space, let $f\in S_{\Lip_0(M)}$, let $\varepsilon>0$, let $r>0$, and let $x,y\in M$ with $x\neq y$ satisfy $\langle f,m_{x,y}\rangle>1-\varepsilon$ and $r<d(x,y)$. For every $\gamma:[0,1]\rightarrow M$ with $\gamma(0)=x$ and $\gamma(1)=y$, we have 
\[
\langle f,m_{u,v}\rangle
>1-\frac{L(\gamma)-(1-\varepsilon)d(x,y)}{r}
\]
for any $u,v\in \image \gamma$ satisfying $d(u,v)\geq r$ and $d(x,u)<d(x,v)$.
\end{lem}
\begin{proof}
The proof follows from a decomposition argument of 
$\delta(x)-\delta(y)$ along the path $\gamma$. Fix $u,v\in \image \gamma$ satisfying $d(u,v)\geq r$ and $d(x,u)<d(x,v)$.
Since
\[
\delta(x)-\delta(y)
=\big(\delta(x)-\delta(u)\big)+\big(\delta(u)-\delta(v)\big)+\big(\delta(v)-\delta(y)\big),
\]
we have
\begin{align*}
\langle f,m_{x,y}\rangle
&=\frac{\langle f,\delta(x)-\delta(y)\rangle}{d(x,y)}\\
&=\frac{\langle f,\delta(x)-\delta(u)\rangle}{d(x,y)}
 +\frac{\langle f,\delta(u)-\delta(v)\rangle}{d(x,y)}
 +\frac{\langle f,\delta(v)-\delta(y)\rangle}{d(x,y)}\\
&=\frac{d(x,u)}{d(x,y)}\,\langle f,m_{x,u}\rangle
 +\frac{d(u,v)}{d(x,y)}\,\langle f,m_{u,v}\rangle
 +\frac{d(v,y)}{d(x,y)}\,\langle f,m_{v,y}\rangle .
\end{align*}
Since $\|f\|=1$, we have $\langle f,m_{x,u}\rangle\le 1$ and
$\langle f,m_{v,y}\rangle\le 1$. Therefore,
\[
1-\varepsilon
<\langle f,m_{x,y}\rangle
\le
\frac{d(x,u)}{d(x,y)}
+\frac{d(u,v)}{d(x,y)}\,\langle f,m_{u,v}\rangle
+\frac{d(v,y)}{d(x,y)}.
\]
Rearranging gives
\[
\langle f,m_{u,v}\rangle
>
(1-\varepsilon)\frac{d(x,y)}{d(u,v)}
-\frac{d(x,u)+d(v,y)}{d(u,v)}.
\]
Finally, because $d(u,v)\ge r$ and
\[
d(x,u)+d(v,y)
\le L(\gamma)-d(u,v),
\]
we obtain
\[
\langle f,m_{u,v}\rangle
>
1-\frac{L(\gamma)-(1-\varepsilon)d(x,y)}{d(u,v)}
\ge
1-\frac{L(\gamma)-(1-\varepsilon)d(x,y)}{r}.
\]
\end{proof}

We are now ready to prove the main result of this note.

\begin{thm}\label{thm:ASD2P_main}
    Let $M$ be a length space. Then $\mathcal{F}(M)$ has the ASD$2$P. 
\end{thm}
\begin{proof}
Assume $M$ is a length space. 
    Let $n\in \N$, $\lambda_1,\ldots,\lambda_n>0$ with $\sum_i\lambda_i=1$,
    let $f_1,\ldots,f_n\in S_{\Lip_0(M)}$, and let $\alpha>0$. For each $i\in \{1,\ldots,n\}$ consider the slice $S_i=S(f_i,\alpha)$. It suffices to find $\mu,\nu\in \sum_{i=1}^n \lambda_iS_i$ so that $\|\mu-\nu\|=2$.

Fix $i\in \{1,\ldots,n\}$ and let 
\[\varepsilon=\frac{\alpha}{18n}.\]  
Find $u_i,v_i\in M$ with $u_i\neq v_i$ such that
$\langle f_i,m_{u_i,v_i}\rangle>1-\varepsilon$. Denote
\[
r_i=\frac{d(u_i,v_i)}{9n}
\]
and assume without loss of generality that $r_1\leq\ldots\leq r_n$. 
Since $M$ is length, we choose a path $\gamma_i:[0,1]\to M$ joining $u_i$ to $v_i$
such that
\[L(\gamma_i)<(1+\varepsilon)d(u_i,v_i)=\alpha r_i+(1-\varepsilon)d(u_i,v_i),\]
implying
\begin{equation}
\frac{L(\gamma_i)-(1-\varepsilon)d(u_i,v_i)}{r_i}<\alpha.
\end{equation}

Next, we inductively construct points $x_i,y_i\in\mathrm{Im}\,\gamma_i$
so that
\begin{equation}
\min\{d(x_i,x_j),d(x_i,y_j),d(y_i,y_j):\, i>j\}\ge 2r_i
\end{equation}

and

\begin{equation}\tag{2.3}
\min\{d(x_i,v_i),d(y_i,u_i),d(x_i,y_i):\, i=1,\dots,n\}\ge 2r.
\end{equation}

For $i=1$ we set $x_1=u_1$ and $y_1=v_1$.

Let $i\in \{2,\ldots,n\}$ and assume that we have defined points $x_j,y_j\in M$ for every $j<i$. 
To choose $x_i$, we consider the set
\[
F_i=\{x_1,y_1,\dots,x_{i-1},y_{i-1},v_i\}.
\]
We claim that
\[
\mathrm{Im}\,\gamma_i
\setminus\bigcup_{p\in F_i} B(p,2r_i)
\neq\emptyset,
\]
Indeed, if $\mathrm{Im}\,\gamma_i\subset \bigcup_{p\in F_i}B(p,2r_i)$, then
\begin{align*}
 L(\gamma_i)
\le \sum_{p\in F_i}\mathrm{diam}(B(p,2r_i))
\le |F_i|\cdot 4r_i
\le
8nr_i
= \frac89 d(u_i,v_i),
\end{align*}
contradicting $L(\gamma_i)\geq d(u_i,v_i)$.
Thus we may choose $x_i\in\mathrm{Im}\,\gamma_i$
with $d(x_i,F_i)\ge 2r_i$.

Next we define
\[
G_i=\{x_1,y_1,\dots,x_{i-1},y_{i-1},x_i,u_i\}.
\]
The same argument shows that
\[
\mathrm{Im}\,\gamma_i
\setminus\bigcup_{p\in G_i} B(p,2r_i)
\neq\emptyset,
\]
so we may choose $y_i\in\mathrm{Im}\,\gamma_i$
with $d(y_i,G_i)\ge 2r_i$.
This completes the induction and establishes (2.2) and (2.3).

Fix $i\in \{1,\ldots,n\}$. 
By $(2.3)$, we can choose points $z_i,w_i\in\mathrm{Im}\,\gamma_i$
such that
\[
d(x_i,z_i)=r_i,
\qquad
d(z_i,v_i)<d(x_i,v_i),
\]
and
\[
d(y_i,w_i)=r_i,\qquad d(w_i,u_i)<d(y_i,u_i).\]
Applying Lemma~2.1 to the path $\gamma_i$
(with $x=u_i$, $y=v_i$, and $r_i$ as above),
and using (2.1), we obtain
\[
\langle f_i,m_{x_i,z_i}\rangle>1-\alpha
\quad\text{and}\quad
\langle f_i,m_{w_i,y_i}\rangle>1-\alpha.
\]
Hence $m_{x_i,z_i},\, m_{w_i,y_i}\in S_i.$

Now we are ready to define
\[
\mu=\sum_{i=1}^n\lambda_i m_{x_i,z_i}
\quad\text{and}\quad
\nu=\sum_{i=1}^n\lambda_i m_{w_i,y_i}.
\]
Then $\mu,\nu\in\sum_{i=1}^n\lambda_i S_i$.

Define $g:M\to\mathbb R$ by
\[
g(p)=\max\bigl\{ r_i-d(x_i,p),\ r_i-d(y_i,p): i=1,\dots,n\bigr\}.
\]
Each function $p\mapsto r_i-d(x_i,p)$ and $p\mapsto r_i-d(y_i,p)$
is $1$-Lipschitz, hence so is $g$. Set $\tilde g(p)=g(p)-g(0)$. Then 
$\tilde g\in S_{\operatorname{Lip}_0(M)}$.

By (2.2) and (2.3), all of the balls $B(x_i,r)$ and $B(y_i,r)$ are pairwise disjoint.
Moreover,
\[
g(x_i)=r_i,\quad g(z_i)=0,
\quad
g(y_i)=r_i,\quad g(w_i)=0.
\]
Therefore,
\[
\langle \tilde g,m_{x_i,z_i}\rangle=1,
\qquad
\langle \tilde g,m_{w_i,y_i}\rangle=-1,
\]
and consequently,
\[\|\mu-\nu\|\geq
\langle \tilde g,\mu-\nu\rangle
=\sum_{i=1}^n\lambda_i(1-(-1))=2.
\]

\end{proof}
\begin{rem}
    Small modifications to the previous proof imply that every convex combination of slices of $B_{\mathcal{F}(M)}$ contains a square point (see \cite[Definition 3.1]{MR4787365}). 
    
    We can require that (2.2) and (2.3) hold with $3r_i$. It then follows that we define $\mu,\nu\in \sum \lambda_i S_i$ such that $\|\mu\pm \nu\|=2$ and thus 
    \[
    \left\|\frac{\mu+\nu}{2}+ \frac{\mu-\nu}{2}\right\|=\left\|\frac{\mu+\nu}{2}- \frac{\mu-\nu}{2}\right\|=1.
    \]
\end{rem}
\section*{Acknowledgements}
The author is grateful to his supervisors Rainis Haller  and Andre Ostrak for helpful comments on the paper. The author is thankful to the anonymous referee for helpful feedback.  

This work was supported by the Estonian Research Council grant (PRG1901).


\bibliographystyle{amsplain}
\bibliography{references}

\end{document}